\documentclass[12pt, reqno]{amsart}

\usepackage{amsmath, amsthm, amscd, amsfonts, amssymb, graphicx, color}
\usepackage[bookmarksnumbered, colorlinks, plainpages]{hyperref}
\usepackage{amsthm, amssymb, amsmath, mathrsfs, graphicx, xcolor, longtable}

\makeatletter

\newcommand{\GF}[1]{\mathbb{F}_{#1}}
\newcommand{\C}[1]{\mathop{\operator@font C \kern 0pt}_{#1}}
\newcommand{\D}[1]{\mathop{\operator@font D \kern 0pt}_{#1}}
\newcommand{\Sym}[1]{\mathop{\operator@font S \kern 0pt}_{#1}}
\newcommand{\Alt}[1]{\mathop{\operator@font A \kern 0pt}_{#1}}
\newcommand{\GL}[2]{\mathop{\operator@font GL}(#1,#2)}
\newcommand{\SL}[2]{\mathop{\operator@font SL}(#1,#2)}
\newcommand{\PGL}[2]{\mathop{\operator@font PGL}(#1,#2)}
\newcommand{\PSL}[2]{\mathop{\operator@font PSL}(#1,#2)}
\newcommand{\PSU}[2]{\mathop{\operator@font PSU}(#1,#2)}
\newcommand{\PSp}[2]{\mathop{\operator@font PSp}(#1,#2)}
\newcommand{\Sz}[1]{\mathop{\operator@font Sz}(#1)}
\newcommand{\N}[2]{\mathop{\operator@font N \kern 0pt}_{#1}(#2)}
\makeatother

\makeatletter \oddsidemargin.9375in \evensidemargin \oddsidemargin
\numberwithin{equation}{section}

\theoremstyle{plain}      \newtheorem{CAGroupClassificationTheorem}
                                     {Theorem}[section]
\theoremstyle{plain}      \newtheorem{MainTheorem}
                                     [CAGroupClassificationTheorem]{Main Theorem}

\theoremstyle{plain}      \newtheorem{DicksonsLemma}
                                     {Lemma}[section]
\theoremstyle{plain}      \newtheorem{MinimalNonCA_PSLsLemma}
                                     [DicksonsLemma]{Lemma}
\theoremstyle{plain}      \newtheorem{NonSolvableMinimalNonCA_PerfectLemma}
                                     [DicksonsLemma]{Lemma}
\theoremstyle{plain}      \newtheorem{SuzukiGroupsLemma}
                                     [DicksonsLemma]{Lemma}

\begin{document}

\setcounter{page}{1}

\title[Finite non-solvable minimal non-CA-groups]
      {A classification of the finite non-solvable minimal non-CA-groups}

\author{Leyli Jafari, Stefan Kohl and Mohammad Zarrin}

\address{Department of Mathematics, 
         University of Kurdistan, 
         P.O. Box: 416, 
         Sanandaj, Iran}

\email{l.jafari@sci.uok.ac.ir, zarrin@ipm.ir}



\begin{abstract}
  A group is called a \emph{\textbf{CA}-group} if the centralizer 
  of every non-central element is abelian. Furthermore, a group
  is called a \emph{minimal non-\textbf{CA}-group} if it is not a
  \textbf{CA}-group itself, but all of its proper subgroups are.
  In this paper, we give a classification of the finite non-solvable
  minimal non-\textbf{CA}-groups.
\end{abstract}

\maketitle

%%%%%%%%%%%%%%%%%%%%%%%%%%%%%%%%%%%%%%%%%%%%%%%%%%%%%%%%%%%%%%%%%%%%%%%%%%%%%%%%%%%%%%%%%%%%%%%%%%%%

\section{Introduction}

Let $\mathcal{X}$ be a class of groups.
A group $G$ is called a \emph{minimal non}-$\mathcal{X}$ group or an $\mathcal{X}$-\emph{critical}
group if $G \notin \mathcal{X}$, but all proper subgroups of $G$ belong to $\mathcal{X}$.
The minimal non-$\mathcal{X}$ groups have been studied for many classes of groups $\mathcal{X}$.
For instance, Miller and Moreno~\cite{Miller-Moreno03} investigated minimal non-abelian groups,
while Schmidt~\cite{Schmidt24} studied minimal non-nilpotent groups. The latter are known as
\emph{Schmidt groups} -- see also~\cite{Zarrin12}. Furthermore, Doerk~\cite{Doerk66} determined the
structure of minimal non-supersolvable groups. 

We say that a group $G$ is a \emph{\textbf{CA}-group} if the centralizer of every non-central
element is abelian. Examples of infinite \textbf{CA}-groups are free groups and Tarski monsters
groups. Finite \textbf{CA}-groups are of historical importance in the context of the classification
of finite simple groups, and there is a complete classification of finite \textbf{CA}-groups by
Schmidt~\cite{Schmidt70} as follows:

\begin{CAGroupClassificationTheorem} \label{CAGroupClassificationTheorem}
  Let $G$ be a non-abelian group, and let $Z$ denote the center of~$G$.
  Then $G$ is a \textbf{CA}-group if and only if one of the following holds:
  \begin{enumerate}

    \item $G$ is non-abelian and has an abelian normal subgroup of prime index.

    \item $G/Z$ is a Frobenius group with Frobenius kernel $K/Z$ and 
          Frobenius complement $L/Z$, where $K$ and $L$ are abelian.

    \item $G/Z$ is a Frobenius group with Frobenius kernel $K/Z$ and Frobenius complement $L/Z$,
          such that $K = PZ$, where $P$ is a normal Sylow $p$-subgroup of $G$ for some prime~$p$
          dividing the order of~$G$, the group $P$ is a \textbf{CA}-group,
          ${\rm Z}(P) = P \cap Z$ and $L = HZ$, where $H$ is an abelian $p'$-subgroup of~$G$.

    \item $G/Z\cong \Sym{4}$, and if $V/Z$ is the Klein four group in $G/Z$,
          then $V$ is non-abelian.

    \item $G = P \times A$, where $P$ is a non-abelian \textbf{CA}-group of prime power order
          and $A$ is abelian.

    \item $G/Z \cong \PSL{2}{p^n}$ or $G/Z \cong \PGL{2}{p^n}$, and $G' \cong \SL{2}{p^n}$, 
          where $p$ is a prime and $p^n > 3$.

    \item $G/Z \cong \PSL{2}{9}$ or $G/Z \cong \PGL{2}{9}$, and $G'$ is isomorphic 
          to the Schur cover of $\PSL{2}{9}$.

  \end{enumerate}
\end{CAGroupClassificationTheorem}

In this paper, we classify the non-solvable finite minimal non-\textbf{CA}-groups -- 
that is, the finite non-solvable groups which are not \textbf{CA}-groups themselves, 
but all of whose proper subgroups are:

\begin{MainTheorem} \label{main}
  A finite group $G$ is a non-solvable minimal non-\textbf{CA}-group 
  if and only if $G \cong \PSL{2}{q}$, and one of the following holds:
  \begin{enumerate}
    \item $q > 5$ is a prime number, and $16 \nmid q^2-1$.
    \item $q = 3^p$, where $p$ is an odd prime number.
    \item $q = 5^p$, where $p$ is an odd prime number.
  \end{enumerate}
\end{MainTheorem}

%%%%%%%%%%%%%%%%%%%%%%%%%%%%%%%%%%%%%%%%%%%%%%%%%%%%%%%%%%%%%%%%%%%%%%%%%%%%%%%%%%%%%%%%%%%%%%%%%%%%

\section{The proof of the main theorem}

To prove the main result, we need the following lemmata.

\begin{DicksonsLemma} \label{DicksonsLemma}
  (Dickson's Theorem; cf. e.g.~\cite{Huppert67}, p.~213): Given a prime power $q = p^m$, 
  a complete list of conjugacy classes of maximal subgroups of $\PSL{2}{q}$ is as follows: 
  \begin{enumerate}

    \item One class of subgroups isomorphic to $\C{q} \rtimes \C{(q-1)/(2,q-1)}$;
    \item One class of subgroups isomorphic to $\D{2(q-1)/(2,q-1)}$,
          in the case that $q \notin \{5, 7, 9, 11\}$;
    \item One class of subgroups isomorphic to $\D{2(q+1)/(2,q-1)}$,
          in the case that $q \notin \{7, 9\}$;
    \item Two classes of subgroups isomorphic to $\Alt{5}$, 
          if $q \equiv \pm 1$ (mod~10), and $\GF{q} = \GF{p}[\sqrt{5}]$;
    \item Two classes of subgroups isomorphic to $\Sym{4}$, if $q^2 \equiv 1$ (mod~16), $q$ prime;
    \item One class of subgroups isomorphic to $\Alt{4}$, if $m = 1$ and $q$ is congruent to either
          3, 5, 13, 27 or 37 (mod~40);
    \item Two classes of subgroups isomorphic to $\PGL{2}{p^{m/2}}$, in the case that $p$ is odd
          and $m$ is even;
    \item One class of subgroups isomorphic to $\PSL{2}{p^n}$, 
          where $m/n$ is an odd prime or $p = 2$ and $n = m/2$.

  \end{enumerate}
\end{DicksonsLemma}

\begin{MinimalNonCA_PSLsLemma} \label{MinimalNonCA_PSLsLemma}
  A projective special linear group $\PSL{2}{q}$ is a minimal non-\textbf{CA}-group
  if and only if one of the following holds:
  \begin{enumerate}
    \item $q > 5$ is a prime number, and $16 \nmid q^2-1$.
    \item $q = 3^p$, where $p$ is an odd prime number.
    \item $q = 5^p$, where $p$ is an odd prime number.
  \end{enumerate}
\end{MinimalNonCA_PSLsLemma}
\begin{proof}
  Let $p$ be a prime, let $m$ be a positive integer, and put $q := p^m$ and $G := \PSL{2}{q}$.
  If $p = 2$, then $G$ is a \textbf{CA}-group by Theorem~\ref{CAGroupClassificationTheorem}.
  Therefore from now on we assume that $p > 2$. 

  Since the non-\textbf{CA}-group $\Sym{4}$ embeds into $G$ if and only if $q^2 \equiv 1$ (mod~16),
  the group $G$ can only be a \textbf{CA}-group if $16 \nmid q^2-1$.

  Suppose that $p > 5$. If $m \neq 1$, then by Lemma~\ref{DicksonsLemma}, the group $G$ has
  a proper subgroup isomorphic to $\PSL{2}{p}$. By Theorem~\ref{CAGroupClassificationTheorem},
  the latter is not a \textbf{CA}-group. Therefore we can assume that $m = 1$.

  Suppose that $p \in \{3, 5\}$. If $m$ is not a prime, then $G$ has a subgroup isomorphic 
  to $H = \PSL{2}{p^r}$ for some $r \mid m$ and $r > 1$, and $H$ is not a \textbf{CA}-group. 
  Therefore in this case $G$ can only be a \textbf{CA}-group if $m$ is a prime number.

  Now let $16 \nmid q^2-1$, where $q$ is either a prime number greater than 5 or a power
  of 3 or~5 with prime exponent. In this case, we conclude from Lemma~\ref{DicksonsLemma} and
  Theorem~\ref{CAGroupClassificationTheorem} that it is enough to show that every subgroup
  isomorphic to $\C{q} \rtimes \C{(q-1)/(2,q-1)}$ is a \textbf{CA}-group -- as it is easy
  to see that all other subgroups of $G$ are \textbf{CA}-groups. For this, we show that these
  subgroups can be constructed in the following way:

  We define two types of elements $d_a$ ($a \in \GF{q}^*$) and $t_b$
  ($b \in \GF{q}$) of the special linear group $\SL{2}{q}$ as follows: 
  \[
    d_a \ := \
   \left[\begin{array}{cc} a & 0 \\ 0 & a ^{-1} \end{array} \right]
  \]
  and
  \[
    t_b \ := \
    \left[\begin{array}{cc} 1 & b \\ 0 & 1  \end{array} \right].
  \]
  Put $D := \{d_a \ | \ a \in \GF{q}^*\}$, $T := \{t_b \ | \ b \in \GF{q}\}$ and $L := DT$.
  Then $D$, $T$ and $L$ are subgroups of $\SL{2}{q}$, and $T$ is normal in~$L$. Clearly
  $$
    L \ = \ \left\{
              \left[ \begin{array}{cc} a & b \\ 0 & a^{-1} \end{array} \right]
              \ \middle | \ a \in \GF{q}^*, \ b \in \GF{q}
            \right\}.
            \hspace{6mm} (*)
  $$
  Now put $\bar{L} := L/Z$, where $Z$ is the center of $\SL{2}{q}$. It is not hard to see that
  $\bar{L}$ is a Frobenius group with complement and kernel $D/Z$ and $TZ/Z$, respectively.
  Hence by Theorem~\ref{CAGroupClassificationTheorem}, the group $\bar{L}$ is a \textbf{CA}-group.
\end{proof}

\begin{NonSolvableMinimalNonCA_PerfectLemma} 
\label{NonSolvableMinimalNonCA_PerfectLemma}
  Let $G$ be a finite non-solvable minimal non-\textbf{CA}-group.
  Then $G$ is perfect.
\end{NonSolvableMinimalNonCA_PerfectLemma}
\begin{proof}
  Suppose that $G$ would be not perfect. 
  If $H' < H$ for every proper subgroup $H$ of $G$, then $G$ is solvable, 
  contrary to our assumption (note that $G$ is a finite group). 
  So we may assume that $H' = H$ for some proper subgroup $H$ of $G$. 
  Now the group $K := \langle H < G | H' = H \rangle$ is perfect,
  and it is a normal subgroup of $G$. By Theorem~\ref{CAGroupClassificationTheorem},
  the only perfect \textbf{CA}-groups are $\SL{2}{q}$ and the Schur cover of $\PSL{2}{9}$. 

  First suppose that $K \cong \SL{2}{q}$ for some prime power $q$.
  Put $Z := {\rm Z}(G)$ and $C := {\rm C}_{G}(K)$.
  We claim that $C = Z$. Assume, on the contrary, that $Z < C$.
  Given $y \in G$, put $H_y := KC\langle y \rangle \leq G$.

  First consider the case that $H_y$ is a proper subgroup of~$G$, for every $y \in G$.
  In this case, by assumption, the $H_y$ are \textbf{CA}-groups. Therefore by
  Theorem~\ref{CAGroupClassificationTheorem}, we have $H_y/{\rm Z}(H_y) \cong \PSL{2}{q}$.
  Since $K \cong \SL{2}{q}$, we hence have $y \in {\rm Z}(H_y)$, so in particular $y$ commutes
  with all elements of~$C$. Now since by assumption all $H_y$ are proper subgroups
  of~$G$, this implies that all $y$ commute with all elements of~$C$, or in other words,
  that $C \leq Z$. It follows that $C = Z$, as claimed.

  Now consider the case that $G = KCY$, where $Y = \langle y \rangle$ for some element $y \in G$ 
  and $Y \nleq C$. Let $b \in \GF{p}^*$, and let $L$ and $t_b$ be the subgroup and the element
  of~$G$ defined in the proof of Lemma~\ref{MinimalNonCA_PSLsLemma} (see~(*)), respectively.
  Now we can assume that conjugation by $y$ induces a field automorphism of
  $K/{\rm Z}(K) \cong \PSL{2}{q}$. Since any field automorphism fixes the prime field
  and since for odd $q$ we have $\PSL{2}{q} \cong \SL{2}{q}/\langle {-1} \rangle$,
  we see that $y$ normalizes $L$ and that conjugation by~$y$ either maps $t_b$
  to itself or to $-t_b$. Clearly $t_b \notin {\rm Z}(LY)$, so $t_b \notin {\rm Z}(LYC)$.
  If $Y = \langle y \rangle$ does not centralize~$t_b$, then $t_b^y = -t_b$.
  But $|{-t_b}| = 2p \neq |t_b| = p$, a contradiction.
  If $Y$ centralizes $t_b$, then $YC \leq {\rm C}_{H_yYC}(t_b)$ is abelian, 
  so $YC$ is abelian. Therefore $C \leq Z$, and so $C = Z$.

  As we have $C = Z$, the group $G/Z$ has an element $g$ conjugation by which induces a field
  automorphism or a field-diagonal automorphism of $K/{\rm Z}(K) \cong \PSL{2}{q}$.
  -- Otherwise we would have $G/Z \cong \PGL{2}{q}$, which implies by
  Theorem~\ref{CAGroupClassificationTheorem} that $G$ is a \textbf{CA}-group -- a contradiction.
 Let $L$ be as above. Then the group $\langle g \rangle$
  acts on~$L$ by conjugation. Now the group $S := L \langle g \rangle$ is a proper subgroup of~$G$
  which by Theorem~\ref{CAGroupClassificationTheorem} is not 
  a \textbf{CA}-group.

  Now suppose that $K$ is the Schur cover of $\PSL{2}{9} \cong \Alt{6}$. 
  It is obvious that $G/C$ is an almost simple group with socle $KC/C \cong \Alt{6}$.
 
  First we claim that $G/CK$ does not contain an element which induces the field automorphism
  of $\PSL{2}{9}$. Assume, on the contrary, that $G/CK$ contains an element which corresponds
  to the field automorphism of $\PSL{2}{9} \cong \Alt{6}$. Then the group $G/C$ has a normal
  subgroup isomorphic to~$\Sym{6}$. Therefore in particular it has a subgroup isomorphic to the
  standard wreath product $\Sym{3} \wr \C{2}$ of $\Sym{3}$ and $\C{2}$. 
  Let $T$ be a subgroup of $G$ such that $T/C \cong \Sym{3} \wr \C{2}$. 
  Since $T$ is a \textbf{CA}-group, one can conclude that $T$ is as in
  Theorem~\ref{CAGroupClassificationTheorem}, Part~(1).
  But such a group does not have a factor isomorphic to $\Sym{3} \wr \C{2}$ --
  a contradiction. Hence our claim is proved. 

  So we have either $G/KC \cong \langle \sigma \rangle$ or 
  $G/KC \cong \langle \sigma\phi \rangle$, where $\phi$ corresponds to the field automorphism 
  and $\sigma$ corresponds to the diagonal automorphism.

  In the former case we have $G/C \cong \PGL{2}{9}$. Let $T < G$ be a subgroup such that $T/C$
  is isomorphic to the dihedral group of order 8 or~10. 
  By Theorem \ref{CAGroupClassificationTheorem}, we conclude that $C = Z$, and that
  $G$ is a \textbf{CA}-group -- and therefore not a minimal non-\textbf{CA}-group.

  In the latter case we have $G/C \cong {\rm M}_{10}$, where ${\rm M}_{10}$ is the Mathieu group
  of degree~$10$. Let $T < G$ be a subgroup such that $T/C \cong \C{3}^2 \rtimes {\rm Q}_8$
  is a Frobenius group of order~$72$. So $T$ is as in Theorem~\ref{CAGroupClassificationTheorem},
  Part~(2) or~(3). But the Frobenius complement ${\rm Q}_8$ of $T/C$ is not abelian -- and hence
  $T$ is not a \textbf{CA}-group, and $G$ is not a minimal non-\textbf{CA}-group.
\end{proof}

\begin{SuzukiGroupsLemma} \label{SuzukiGroupsLemma}
  The Suzuki group $\Sz{2^{2n+1}}$, $n \geq 1$, is not a minimal 
  non-\textbf{CA}-group.
\end{SuzukiGroupsLemma}
\begin{proof}
  Let $G := \Sz{r}$, where $r = 2^{2n+1}$ for some $n \geq 1$.

  If $2n+1$ is prime, then by~\cite{Huppert-Blackburn82}, Theorem~3.3, the group $G$ is 
  a Zassenhaus group of degree $r^2+1$. Let $N < G$ be a point stabilizer. 
  Then $N$ is a Frobenius group with a non-abelian kernel $K$ of order $r^2$ 
  and a cyclic complement $H$. Hence, by
  Theorem~\ref{CAGroupClassificationTheorem}, $N$ is not a \textbf{CA}-group, 
  and therefore $G$ is not a minimal non-\textbf{CA}-group.

  If $2n+1$ is not prime, then $G$ has a subgroup isomorphic to $\Sz{r_0}$, 
  where $r_0 = 2^s$ for some prime number $s$, which has a subgroup that is 
  not a \textbf{CA}-group.
\end{proof}

Now we are ready to prove the main result.

\vspace{2mm}

\noindent \textbf{Proof of Theorem 1.2}.
Let $G$ be a non-solvable minimal non-\textbf{CA}-group.
By Lemma~\ref{NonSolvableMinimalNonCA_PerfectLemma}, the group $G$ is perfect.
Let $N$ be the solvable radical -- i.e. the largest solvable normal subgroup -- of $G$,
and let $M/N$ be a minimal normal subgroup of $G/N$.
First assume that $G/N \neq M/N$. Assume that $M/N = (T/N)^k$, 
where $T/N$ is a simple group. First we prove that $k = 1$. 
On the contrary, assume that $k \geq 2$. Since $T$ is a \textbf{CA}-group, 
we have $T/N \cong \PSL{2}{q}$ and $N \subseteq {\rm Z}(T)$. 
Now take a non-abelian subgroup of $T$, say $H$. Then obviously the subgroup 
generated by $H$ and one other copy of $T$ is not a \textbf{CA}-group --
a contradiction. So we see that $M/N$ is isomorphic to $\PSL{2}{q}$ 
and $N \subseteq {\rm Z}(M)$. By similar arguments, we have ${\rm C}_{G/N}(M/N) = 1$,
and hence $G/N$ is an almost simple group with socle $M/N$.
Since $G$ is perfect, we get a contradiction, and hence $G/N$ is simple. \\
If all subgroups of $G$ are solvable, then $G$ is a minimal non-solvable group.
In this case, by~\cite{Thompson68}, the Frattini quotient $G / \Phi(G)$ is
isomorphic to either a Suzuki group or a projective special linear group.

Now we claim that $\Phi(G) = {\rm Z}(G)$. 
Let $T$ be a subgroup of~$G$ such that $T/N$ is a dihedral group of order 
$q-1$ or $q+1$ or $2(q-1)$. Then $T$ is a solvable \textbf{CA}-group, so it falls under
one of the cases (1) -- (5) in Theorem~\ref{CAGroupClassificationTheorem}.
We conclude that $N < {\rm C}_G(N)$, and so $N = {\rm Z}(G)$.

As $G$ is perfect, $N$ is contained in the Schur multiplier of $G/N$.
Therefore $G$ is isomorphic to either a special linear group, $3.\Alt{6}$, $6.\Alt{6}$,
a projective special linear group, or a Suzuki group. By Lemma~\ref{SuzukiGroupsLemma}
and~\ref{MinimalNonCA_PSLsLemma}, among these groups, only the projective special linear groups
$\PSL{2}{q}$, for some particular values of~$q$, are minimal non-\textbf{CA}-groups.

Now assume that $G$ has at least one non-solvable subgroup, and that also all 
non-solvable subgroups of $G$ have a factor isomorphic to either $\PSL{2}{q}$ or $\PGL{2}{q}$.
It is not hard to see that for every non-solvable subgroup $H/N$ of $G/N$, the quotient
$(H/N)/{\rm Z}(H/N)$ is isomorphic to either $\PSL{2}{q}$ or $\PGL{2}{q}$. 
With a case-by-case analysis we can see that the only simple groups which satisfy
these conditions are projective special linear groups $\PSL{2}{q}$ for certain 
prime powers~$q$. Now Lemma~\ref{MinimalNonCA_PSLsLemma} completes the proof. \qed

%\section*{Acknowledgments}
%The authors would like to express their gratitude to Silvio Dolfi and Zeinab Akhlaghi for 
%invaluable comments. 

%%%%%%%%%%%%%%%%%%%%%%%%%%%%%%%%%%%%%%%%%%%%%%%%%%%%%%%%%%%%%%%%%%%%%%%%%%%%%%%%%%%%%%%%%%%%%%%%%%%%

\end{document}